\documentclass[11pt,tbtags]{amsart}
\usepackage{amssymb}
\usepackage{amsmath}
\usepackage{amsthm}
\usepackage{enumitem}
\usepackage{dsfont}
\usepackage{float}

\usepackage[swedish, english]{babel}
\usepackage[latin1]{inputenc}
\usepackage{psfrag}   
\usepackage{graphicx}
\usepackage[square, comma, numbers, sort&compress]{natbib}
\usepackage{enumitem}
\usepackage{color}

\usepackage{caption}
\usepackage{subcaption}	
\usepackage[percent]{overpic}
\usepackage{pict2e}

\theoremstyle{plain}

\newcommand{\Cov}{\mathrm{Cov}}

\newcommand{\E}{\mathbb E}
\newcommand{\R}{\mathbb R}

\newcommand{\D}{\mathcal D}
\newcommand{\C}{\mathcal C}

\def\P{{\mathbb P}}

\newenvironment{example}[1][Example]{\begin{trivlist}
\item[\hskip \labelsep {\bf Example}]}{\end{trivlist}}

\newtheorem{theorem}{Theorem}[section]
\newtheorem{lemma}[theorem]{Lemma}
\newtheorem{corollary}[theorem]{Corollary}

\newtheorem{assumption}[theorem]{Assumption}
\newtheorem{remark}[theorem]{Remark}
\newtheorem{conjecture}[theorem]{Conjecture}

\DeclareMathOperator\supp{supp}

\title{Bayesian sequential composite hypothesis testing in discrete time}
\author{Erik Ekstr\"om and Yuqiong Wang}
\address{Department of Mathematics, Uppsala University, Box 256, 75105 Uppsala, Sweden.}
\date{\today}

\begin{document}

\begin{abstract} 
We study the sequential testing problem of two alternative hypotheses regarding an unknown parameter in an exponential family when observations are costly. In a Bayesian setting, the problem can be embedded in a Markovian framework. Using 
the conditional probability of one of the hypotheses as the underlying spatial variable, we show that the cost function is concave
and that the posterior distribution becomes more concentrated as time goes on. Moreover, we study time monotonicity of the value function. For a large class of model specifications, the cost function is non-decreasing in time, and the optimal stopping boundaries are thus monotone.
\end{abstract}

\maketitle

\section{Introduction}

Assume that 
a sequence of random variables $X_1, X_2,...$ is observed sequentially, and that the sequence is drawn from a one-parameter family of distributions depending on a real-valued random variable $\Theta$ in such a way that $X_1,X_2,...$ are independent (conditional on $\Theta$). 
Consider a tester who wants to test the two alternative hypotheses 
\begin{align*}
H_0: &\quad \Theta \leq \theta_0,\\
H_1: & \quad\Theta > \theta_0,
\end{align*}
where $\theta_0 \in \R$ is a given constant (the 'threshold'). In the presence of an observation cost, a tradeoff between statistical precision and costly observation arises. 

In a Bayesian formulation of the problem, the tester's initial belief is described by a prior distribution $\mu$ for the unknown parameter $\Theta$.
Denote by $\mathcal{T}$ the set of $\mathbb F^X$-stopping times with values in $\mathbb N_0=\{0,1,2...\}$, where $\mathbb F^X=\{\mathcal F^X_n\}_{n=1}^\infty$ is the filtration generated by the observation process $X=\{X_n,n\geq 0\}$. Given a stopping time 
$\tau\in\mathcal T$, let $\mathcal D^\tau$ be the set of $\mathcal F^X_\tau$-measurable random variables $d$ with values in $\{0,1\}$. 
The random variable $d$ here represents the decision of the tester, with '$d=i$' representing that hypothesis $H_i$ is accepted. We define the cost 
\begin{equation}
\label{V_fcn}
V:=\inf_{\tau\in\mathcal T}\inf_{d\in\mathcal D^\tau}\left\{\P(d=1, \Theta \leq \theta_0)+\P(d=0, \Theta> \theta_0)+c\E[\tau]\right\},
\end{equation}
where $c>0$ is a given and fixed cost of each observation. 

The case when $\mu$ is a two-point distribution with
\begin{align*}
\P(\Theta=\theta_2)&=\pi,\\
\P(\Theta=\theta_1)&=1-\pi,
\end{align*}
where $\pi \in (0,1)$ and $\theta_1\leq\theta_0<\theta_2$ was studied in the classical reference \cite{WW}, see also 
\cite[Chapter 4.1]{S}. It turns out that
the statistical problem \eqref{V_fcn} can be reduced to an optimal stopping problem in terms of the 
posterior probability process $\Pi_n:=\P(\Theta=\theta_2\vert \mathcal F_n)$, and since $\Pi$ in this case is a (time-homogeneous) Markov process, the stopping problem can be embedded in a Markovian framework. It is shown in \cite{S} that the cost function is concave in the prior belief $\Pi_0=\pi$;
as a consequence, the continuation region is an interval, and the optimal stopping time is the first exit time from this interval (the latter property was also obtained in \cite{WW}).

In the current article we relax the assumption about a two-point prior distribution and study the sequential analysis problem \eqref{V_fcn} in a Bayesian set-up for general prior distributions $\mu$. To do that, we impose a one-dimensional 
exponential structure on the distribution of $X_k$. As in \cite{S}, the conditional probability process $\Pi$ is then still Markovian; however,  $\Pi$ is in general time-inhomogenous, which leads to 
time-dependence in the cost function, and the study of optimal strategies is more involved. In the absence of explicit solutions
for the cost and the optimal strategy, we focus on structural properties of the solution. In particular, we prove that spatial concavity of the cost 
function holds regardless of the prior distribution. We also show a concentration result for the posterior distribution, which combined with the concavity result has implications for the monotonicity of the cost with respect to the time parameter.

\subsection{Literature review}
The problem of sequential testing of an unknown parameter has attracted much attention in the statistical literature, with 
\cite{WW} as an early reference covering the case of two simple hypotheses and independent and identically distributed observations. Sequential testing of composite hypotheses in a discrete time setting with Bernoulli distributed observations is studied in \cite{LB} and \cite{MR}, with linear penalty for wrong decisions and relying on a conjugate prior for the unknown
parameter. 
In \cite{MS}, Sobel studies sequential testing of composite hypotheses for an arbitrary class of distributions in the exponential family and with a general prior distribution of the unknown parameter. In a key result, he 
establishes the existence of two stopping boundaries beyond which it is optimal to stop. 
Related literature in discrete time, but more focused on the case of sequential {\em estimation}, includes \cite{A} and \cite{C}.

Another strand of literature has focused on continuous time approximations of sequential testing problems and their connections with free boundary problems.
For the sequential testing of two simple hypotheses, \cite{Sh} solved the problem of determining the unknown drift 
of a Brownian motion, and
\cite{PS} solved the corresponding sequential testing problem of determining an unknown intensity of a Poisson process. 
In \cite{Ba}, a problem with composite hypotheses was studied in continuous time and for a normal prior distribution, with a '0-1' loss function for wrong decisions (as in \eqref{V_fcn}), and in a series of papers (see \cite{Ch} and the references therein), Chernoff studied the same problem but with linear penalty functions.
In the case of sequential composite hypothesis testing, explicit solutions are rare, and a main focus in this literature is on deriving asymptotics of the problem as the cost of observation tends to zero, as well as asymptotically optimal solutions (e.g. \cite{Bi}, \cite{L} and \cite{Sch})
and deriving bounds for the stopping boundaries.

More recent literature has focused on different variants of these continuous-time problems. To mention a few, \cite{GP} studies a version with finite horizon, 
\cite{DS} studies a setting with combined learning from several Brownian motions and compound Poisson processes, and
\cite{EW} studies Wiener sequential testing in a multi-dimensional set-up.
All these papers study simple hypotheses, i.e. set-ups where the unknown parameters can take only two possible values. In \cite{ZS}, a hypothesis testing problem for a case with three possible drifts is examined, and in \cite{EV} a composite hypothesis problem for the drift of a Wiener process is studied with a general prior distribution. 
Moreover, \cite{EKV} study a sequential estimation problem for a Wiener process in the same set-up. Key to the analysis in \cite{EV} and \cite{EKV} is the choice of appropriate variables. In fact, in \cite{EV} it is shown that if instead of the observation process one uses the conditional probability $\Pi$ as state variable, then the corresponding continuation region is shrinking in time; a similar result holds for sequential least-square estimation if one uses the conditional expectation as state variable.

\subsection{Our contribution}
In the current article, we study the sequential composite hypothesis testing problem \eqref{V_fcn} using a Markovian approach. 
Our analysis is general in the sense that we treat the whole one-parameter exponential family with arbitrary prior distribution, and we thus do not rely on 
conjugate priors. Following \cite{EV}, we use the conditional probability process as the underlying state variable, and we show
that a concavity result holds in these coordinates. We also use these coordinates to obtain a concentration result for the posterior distribution, which then is used to show that spatial concavity is intimately connected with monotonicity with respect to time. In particular, we provide a condition under which the continuation region is non-increasing in time. In principle, translating 
back to the observation coordinates, this would give an upper bound on the growth of the stopping boundaries.
%
%The readers might immediately relate to the concept of conjugate priors when they think of distributions in the exponential family. Indeed, when adapting the conjugate priors of the observed random variables with certain distributions, we get a more explicit view of the problem. However, the family of the statistical problem we can formulate with the conjugate priors would be much more smaller.  Under a general prior distribution, as \cite{EE} suggests, we still get useful properties in lack of explicit solutions. We therefore consider the more natural setting with a general family of distributions $\mu$. 

The paper is organised as follows. 
In Section~\ref{inference} we recall some basic properties of statistical inference in the exponential family, and we 
introduce the notion of $\pi$-level curves along which the value of the conditional probability $\Pi$ is constant.
In Section~\ref{Mark}, we provide a Markovian embedding of \eqref{V_fcn}, and we prove that the embedded cost function is spatially concave. In Section~\ref{concentration} we prove that the posterior distribution becomes more concentrated about the threshold $\theta_0$ along level curves. Sections~\ref{mono1}-\ref{mono2} deal with the question whether the value function is monotone with respect to the time parameter.

\section{Preliminaries on the exponential family}
\label{inference}

In this article, we will consider the case of a one-dimensional exponential family of distributions for $X_k$, $k\geq 1$. More precisely, let $\nu$ be a $\sigma$-finite measure $\nu$ on $\R$, and define
\[B(u) := \log \left(\int_{\R} \exp\{u x\}\nu(dx)\right)\]
and 
\[N=\left\{u\in\R: \int_\R\exp\{u x\}\nu(dx)<\infty\right\}\]
so that 
\[B(u) <\infty\]
for $u\in N$. 
For $u\in N$, let
\begin{equation}
\label{p}
p_u(x):= \exp\{u x-B(u)\}
\end{equation}
so that $\int_\R p_u\nu(dx)=1$. We assume that the distribution of $X_k$, 
conditional on $\Theta=u$, is 
\begin{equation}
\label{expfam}
\P(X_k\in A\vert \Theta=u)=\int_A p_u(x)\nu(dx).
\end{equation}

\begin{remark}
In some literature, the notion of an exponential family allows for densities on the form $p_u(x)= \exp\{\eta(u) T(x)-B(u)\}$, and the
case \eqref{expfam} in which $\eta(u)=u$ and $T(x)=x$ is then refered to as a {\em natural} exponential family. Using
the transformed variables $\eta=\eta(u)$ and $T=T(x)$, an exponential form can be transformed into a natural form, so 
we may consider the natural form (as above) without loss of generality.
\end{remark}

We start with some well-known results.

\begin{lemma}
\label{coro2}
We have that
\begin{itemize}
\item[(i)]
$B$ is convex, and $N$ is an interval.
\end{itemize}
Denote by $N^\circ$ the interior of $N$. Then
\begin{itemize}
\item[(ii)]
all derivatives of $B$ exist on $N^\circ$, and they are given by the expressions obtained by formally differentiating inside the integral. In particular, 
\[B'(u)=\frac{\int_{\R} x\exp\{u x\}\nu(dx)}{\int_{\R} \exp\{u x\}\nu(dx)}= \E[X_1\vert \Theta=u];\]
\item[(iii)]
the function $u \mapsto \E[G(X_1)\vert \Theta=u]$ is non-decreasing 
for any non-decreasing function $G: \R \to \R$.
\end{itemize}
\end{lemma}

\begin{proof}
For (i) and (ii) we refer to \cite[Theorem 1.13]{B}) and \cite[Theorem 2.2]{B}, respectively.  For (iii), we have
\begin{eqnarray*}
\frac{\partial}{\partial u} \E[G(X_1)\vert\Theta=u] &=& \frac{\partial}{\partial u}\int_\R G(x) p_u(x)\,\nu(dx)
=\int_\R G(x) (x-B'(u))p_u(x)\,\nu(dx)\\
&=& \E[G(X_1)X_1\vert \Theta=u]-\E[G(X_1)\vert \Theta=u]\E[X_1\vert \Theta=u]\geq 0,
\end{eqnarray*}
where the final inequality is due to the fact the covariance of two non-decreasing functions evaluated at the same random variable is non-negative.
\end{proof}

%
%\begin{lemma}
%
%The function $u \mapsto \E[G(X_1)\vert \Theta=u]$ is non-decreasing 
%for any non-decreasing function $G: \R \to \R$.
%\end{lemma}
%
%\begin{proof}
%We have
%\begin{eqnarray*}
%\frac{\partial}{\partial u} \E[G(X_1)\vert\Theta=u] &=& \frac{\partial}{\partial u}\int_\R G(x) p_u(x)\,\nu(dx)
%=\int_\R G(x) (x-B'(u))p_u(x)\,\nu(dx)\\
%&=& \E[G(X_1)X_1\vert \Theta=u]-\E[G(X_1)\vert \Theta=u]\E[X_1\vert \Theta=u]\geq 0,
%\end{eqnarray*}
%where the final inequality is due to the fact the covariance of two non-decreasing functions evaluated at the same random variable is non-negative.
%\end{proof}

We use a Bayesian set-up in which the unknown parameter $\Theta$ has a given prior distribution $\mu$; we assume that $\mu$ is a measure on $N^{\circ}$, and we denote the support of $\mu$ by $S$. 
Moreover, denote 
\[S^+ = S \cap (\theta_0, \infty)\quad \& \quad S^- = S\cap(-\infty,\theta_0]= S \setminus S^+. \]
Naturally, to avoid degenerate cases we assume that $0<\mu(S^+)<1$.

%Throughout this article, we work under the following standing assumption.
%
%\begin{assumption}
%$\mu$ is a measure on $N^{\circ}$, and we denote the support of $\mu$ by $S$.
%\end{assumption}

%
%The parameter $\theta$ is called the natural parameter, and $T$ is the sufficient statistic for the parameter $\theta$. 
%We consider $\Theta \in \R$ takes value in the set
%\[\mathcal{S} = \{\theta: \int_{\R} h(x)\exp\{\theta T(x)\}dx< \infty\}.\]
%Furthermore, the support of the exponential family remains the same across the parameters. 

%In particular, we will see that there is a one-dimensional sufficient statistic (namely 
%$Y_n=X_1+...+X_n$). Moreover, we translate the statistical problem \eqref{V_fcn} into an optimal stopping problem of the conditional probability process $\Pi$, whi
%
%In our case, the exponential family of distributions for $X_k$ implies a bijection $y\mapsto p(n,y)$, which ensures that we can translate the underlying problem from the $y$ coordinate to the $\pi$ coordinate, and vice versa. 

Next, by standard means, the optimization problem \eqref{V_fcn} can be reduced to an optimal stopping problem, i.e. a problem 
in which only one optimization (namely over $\tau$) takes place. In fact, 
given a stopping time $\tau\in\mathcal T$, an optimal decision rule $d\in\mathcal D^\tau$ is given by
\begin{align*}
d=\begin{cases}
0 & \text{if }  \Pi_\tau \leq 1/2\\
1 & \text{if } \Pi_\tau > 1/2,
\end{cases}
\end{align*}
where the posterior probability process $\Pi$ is given by
\[\Pi_n:=\P(\Theta >\theta_0| \mathcal{F}_n^X).\]
Consequently, 
\[V=\inf_{\tau\in\mathcal T}\E\left[\Pi_\tau\wedge(1-\Pi_\tau) + c\tau\right],\]
where $a\wedge b=\min\{a,b\}$.
To derive an expression for $\Pi_1$, note that
\[
\P(\Theta > \theta_0| X_1=x_1)=\frac{\int_{S^+}p_u(x_1) \mu(du)}{\int_{S}p_u(x_1) \mu(du)},\]
so 
\[\Pi_1=\frac{\int_{S^+}p_u(X_1) \mu(du)}{\int_{S}p_u(X_1) \mu(du)}
.\]
More generally, at time $n$, given observations $X_1 = x_1, X_2 = x_2, \dots, X_n = x_n,$
we have by independence
\begin{align*}
\P(\Theta > \theta_0| X_1=x_1, \dots, X_n=x_n)&=\frac{\int_{S^+}\prod_{i=1}^np_u(x_i) \mu(du)}{\int_{S}\prod_{i=1}^np_u(x_i) \mu(du)}\\
&=\frac{\int_{S^+} \exp\{u\sum_{i=1}^nx_i-nB(u)\} \mu(du)}{\int_{S}
\exp\{u\sum_{i=1}^nx_i-nB(u)\} \mu(du)}.
\end{align*}
Thus, denoting
\[Y_n := \sum_{i=1}^n X_i\] 
we have
\[
\Pi_n= q(n, Y_n),
\]
where 
\[q(n, y) := \frac{\int_{S^+}e^{uy-nB(u)} \mu(du)}{\int_{S}e^{uy-nB(u)}\mu(du)}.\]

\begin{remark}
The fact that $Y_n$ is a sufficient statistic in any exponential family is well-known. Moreover, also a converse holds: under some mild conditions it holds that any 
family of distributions that admits a real-valued sufficient statistic for sample size larger than one is a one-parameter exponential family, see e.g. \cite{Br} and \cite{H}.
\end{remark}

We denote by
\[\mu_{n,y}(du) := \frac{e^{uy-nB(u)} \mu(du)}{\int_{S}e^{uy-nB(u)}\mu(du)}\]
the posterior distribution of $\Theta$ at time $n$ conditional on $Y_n=y$. Note that the prior distribution satisfies $\mu=\mu_{0,0}$; however, for reasons of Markovian embedding, 
below we will consider simultaneously the whole family $\{\mu_{0,y},y\in\R\}$ of alternative prior distributions.

%
%\begin{remark}
%Using the parameter $\eta=\theta$ instead of $\theta$, it can be shown that one without loss of generality may assume that $\theta=\theta$. In view of this, the above assumption is not a severe restriction. However, we choose to keep the formulation in \eqref{expfam} (including the function $\eta$).
%\end{remark}

\begin{lemma}
\label{lem}
The function $y \mapsto q(n,y):\R\to(0,1)$ is an increasing bijection for each fixed $n$.
\end{lemma}

\begin{proof}
We have 
\begin{eqnarray}
\label{cov}
\frac{\partial{q(n,y)}}{\partial y}&=& \frac{\int_{S^+}ue^{uy-nB(u)} \mu(du)}{\int_{S}e^{uy-nB(u)}\mu(du)}- \frac{\int_{S}ue^{uy-nB(u)} \mu(du)\int_{S^+}e^{uy-nB(u)} \mu(du)}{(\int_{S}e^{uy-nB(u)}\mu(du))^2}\\
\notag
&=& \E[\Theta\mathds{1}_{\{\Theta> \theta_0\}} \vert Y_n=y]-\P(\Theta> \theta_0\vert Y_n=y)\E[\Theta\vert Y_n=y].
%&=Cov_{n,y}(\mathds{1}_{\{\Theta> \theta_0\}}, \Theta).
\end{eqnarray}
Since $\mu_{n,y}$ assigns positive mass on each side of the threshold
$\theta_0$, the above covariance is strictly positive. Thus $\frac{\partial{q(n,y)}}{\partial y}>0$,
so $q(n,\cdot)$ is strictly increasing. Moreover, 
\[
\frac{\int_{S^+}e^{uy-nB(u)} \mu(du)}{\int_{S^-}e^{uy-nB(u)}\mu(du)}\geq 
\frac{\int_{S^+}e^{(u-\theta_0)y-nB(u)} \mu(du)}{\int_{S^-}e^{-nB(u)}\mu(du)}\to \infty\]
as $y\to\infty$, so $q(n,y)\to 1$ as $y\to\infty$. A similar argument shows that $q(n,y)\to 0$ as
$y\to-\infty$, so $q(n,\cdot)$ is surjective.
\end{proof}

%From the above, it is clear that $Y_n$ is a sufficient statistic at time $n$. Furthermore, the pair $(n,Y_n)$ does not only characterise the probability $\P(\Theta>\theta_0|\mathcal{F}_n)$, but it also gives us the whole posterior distribution of the unknown parameter $\Theta$. 

For each fixed value $\pi\in (0,1)$, denote by $y(n,\pi)$ the unique value such that $q(n,y(n,\pi))=\pi$.
The set $\{(n,y(n, \pi),n\geq 0\}$ consists of all points $(n,y)$ with $q(n,y)=\pi$, and is refered to as the $\pi-$level curve. Since the function $y\mapsto q(n,y)$ is a bijection, 
two $\pi-$level curves with different $\pi$-values never intersect. Furthermore, they are ordered so that if $\pi_1<\pi_2$, then $y(n,\pi_1)<y(n,\pi_2)$. 

%On the other hand, the values $Y_n$ can take is restricted by the support of $X$, which might not allow $\pi$ to take an arbitrary value in $(0,1)$. However, we can use $y(n,\pi)$ to parametrise the posterior probability measure and allow $y$ to take any value in $\R$. The posterior probability $\pi$ can thereby take any value in $(0,1)$. Specifically, the value $y$ which yields the prior distribution $\pi$ at time $0$ can be regarded as an artificial prior observation, which reflects our belief and can take any real value.

%
% 
%We further observe that, as a conditional expectation, the $\Pi$ process is a Doob martingale. As $n \to \infty$, $\Pi_n$ converges to a Bernoulli random variable $\Pi_{\infty}$ with 
%\begin{align*}
%\P(\Pi_{\infty}=0)&= \mu \left( ( -\infty,\sigma_0]\right),\\
%\P(\Pi_{\infty}=1)&= \mu \left( (\sigma_0, \infty)\right).
%\end{align*}

\section{Markovian embedding}
\label{Mark}

It follows from Lemma~\ref{lem} that the process $\Pi$ is a (time-inhomogeneous) Markov process, and we can write the $\Pi$-process in terms of $Y$ as
\[\Pi_n=\int_{S^+}\mu_{n,Y_n}(du)=\frac{\int_{S^+}p_{u}(X_n) \mu_{n-1,Y_{n-1}}(du)}{\int_{S}p_{u}(X_n)  \mu_{n-1,Y_{n-1}}(du)}.\]
Furthermore, this allows us to embed the optimal stopping problem \eqref{V_fcn} as a time-dependent problem in terms of the Markov process $\Pi$ as 
\begin{equation}\label{V}
V(n, \pi)=\inf_{\tau\in \mathcal{T}}\E_{n,\pi}[\Pi_{\tau+n} \land (1-\Pi_{\tau+n})+c\tau].
\end{equation}
Here $\P_{n,\pi}(\cdot):=\P( \cdot\vert \Pi_n=\pi)$ is the probability measure under which $\Theta$ has distribution $\mu_{n,y(n,\pi)}$.
We emphasize that $V:\mathbb N_0\times(0,1)\to[0,\infty)$, i.e. $\pi$ can take any value in $(0,1)$.

\begin{lemma}
\label{iteration}
The value function $V(n,\pi)$ satisfies
\[V(n-1,\pi)=\min \{ \pi \land (1-\pi), c+ \E_{n-1, \pi}[V(n,\Pi_n)]\}.\]
\end{lemma}

\begin{proof}
This follows directly from the Markovian structure of the process $\Pi$.
\end{proof}

\begin{lemma}
\label{concav}
Let $f: [0,1] \to [0,\infty)$ be a concave function. Then $\pi\mapsto \E_{n,\pi}[f(\Pi_{n+1})]$ is concave on $(0,1)$. 
\end{lemma}

\begin{proof}
To simplify the notation, we prove the statement for $n=0$.
Moreover, we will assume that $f$ is twice continuously differentiable; the general case follows readily by approximation.

%We have
%\[\P(B > \theta_0 )= \int^{\infty}_{\theta_0} \mu_{0,y} (du) = \pi,\]
%where 
%\[\mu_{0,y} (du)= \frac{e^{uy} \mu(du)}{\int_{-\infty}^{\infty} e^{uy} \mu(du)}.\]
First note that
\[\E_{0,\pi}[f(\Pi_1)]= \int_{\R} f \left(\frac{\alpha(x,\pi)}{\beta(x,\pi)} \right)\beta(x,\pi)\,dx,\]
where
\begin{align*}
\alpha(x,\pi)&=\int_{S^+} p_u(x) \mu_{0,y(0,\pi)} (du),\\
\beta(x,\pi)&=\int_S p_u(x) \mu_{0,y(0,\pi)} (du).
\end{align*}
Define
\[
H_1(z):=f(z)+(1-z)f'(z)\]
and 
\[ H_2(z):=f(z)-zf'(z).\]
Straightforward differentiation yields
\begin{align*}
\frac{\partial^2 \E_{0,\pi}[f(\Pi_1)]}{\partial \pi^2}&=\int_{\R} \left( f(\frac{\alpha}{\beta}) \beta_{\pi\pi} 
+ f' (\frac{\alpha}{\beta}) \frac{(\beta \alpha_{\pi\pi}-\alpha \beta_{\pi\pi})^2}{\beta}+ f''(\frac{\alpha}{\beta}) \frac{(\beta \alpha_{\pi}-\alpha \beta_{\pi})^2}{\beta^3} \right)dx\\
&\leq \int_{\R} \left( f(\frac{\alpha}{\beta})\beta_{\pi\pi} + f'(\frac{\alpha}{\beta})  \frac{(\beta \alpha_{\pi\pi}-\alpha \beta_{\pi\pi})^2}{\beta}\right)dx\\
&=\int_{\R} \left( \alpha_{\pi\pi}  H_1\left(\frac{\alpha}{\beta}\right)  +(\beta-\alpha)_{\pi\pi}
H_2\left(\frac{\alpha}{\beta} \right) \right)dx\\
&= I_1+ I_2,
\end{align*}
where 
\[I_1:=\int_{\R}  \alpha_{\pi\pi}  H_1\left(\frac{\alpha}{\beta}\right)  dx\quad \& \quad I_2:=\int_{\R} (\beta-\alpha)_{\pi\pi}H_2\left(\frac{\alpha}{\beta} \right)   dx\]
Note that $H_1$
is decreasing on $(0,1)$, and $H_2$
is increasing. Furthermore, by Lemma~\ref{lem}, $\frac{\alpha(x,\pi)}{\beta(x,\pi)}$ increases in $x$. 

We will show that 
\[I_1 \leq 0\quad \&\quad I_2\leq 0.\]
To do that, first note that 
\[\alpha(x,\pi)=\int_{S^+} p_u(x)\frac{e^{uy(0,\pi)}}{\int_\R e^{uy(0,\pi)}\mu(du)}\mu(du),\]
so
\begin{align}\label{I1}
I_1 %&=\int_{\R} \alpha_{\pi\pi}(x,\pi) H_1\left(\frac{\alpha(x,\pi)}{\beta(x,\pi)}\right) dx\\
%&=\int_{\R} \int^{\infty}_{\theta_0}h(x)e^{ux-B(u)}\left( \mu_{0,y} (du)\right)_{\pi\pi}H_1\left(\frac{\alpha}{\beta}(x)\right) dx\\
\notag
&= \int_{S^+}  \left( \frac{e^{uy(0,\pi)}}{\int_{\R}e^{uy(0,\pi)} \mu(du)}\right)_{\pi\pi}\int_{S}p_u(x)H_1\left(\frac{\alpha(x,\pi)}{\beta(x,\pi)}\right) dx  \mu(du)\\
&=\int_{S^+}  \left( \frac{e^{uy(0,\pi)}}{\int_{\R}e^{uy(0,\pi)} \mu(du)}\right)_{\pi\pi}\E \left[H_1\left(\frac{\alpha(X_1,\pi)}{\beta(X_1,\pi)}\right)\vert\Theta=u\right]     \mu(du).
\end{align}
By Lemma~\ref{coro2}, the function
\begin{equation}\label{decr}
u\mapsto \E \left[H_1\left(\frac{\alpha(X_1,\pi)}{\beta(X_1,\pi)}\right)\vert\Theta=u\right]  
\end{equation}
is non-increasing. 

To study the first factor of the integrand in \eqref{I1}, denote $g(y):= \int_{\R}e^{uy}\mu(du)$ and note that
\[\frac{\partial}{\partial \pi} \left(\frac{e^{uy(0,\pi)}}{\int_{S}e^{uy(0,\pi)} \mu(du)}\right) =\left.\frac{ \frac{\partial}{\partial y} (\frac{e^{uy}}{ g(y)})}{ \pi'(y)}\right\vert_{y=y(0,\pi)}, \]
where 
\[\pi(y):=\frac{\int_{S^+}e^{uy}\mu(du)}{\int_{S}e^{uy}\mu(du)}.\]
Consequently,
\begin{eqnarray*}
\frac{\partial^2}{\partial \pi^2} \left(\frac{e^{uy(0,\pi)}}{\int_{\R}e^{uy(0,\pi)} \mu(du)}\right) & =&\left. \frac{ \pi'(y)\frac{\partial^2}{\partial y^2} (\frac{e^{uy}}{g(y)})-\pi''(y)\frac{\partial}{\partial y} (\frac{e^{uy}}{g(y)})}{ \pi'(y)^3}\right\vert_{y=y(0,\pi)}
\end{eqnarray*}

Using
\[\frac{\partial}{\partial y} \frac{e^{uy}}{g}= \frac{ e^{uy}}{g^2}(ug-g')\]
and 
\[\pi' = \frac{g\int_{S^+}ue^{uy}\mu(du)- g' \int_{S^+}e^{uy}\mu(du) }{g^2},\]
%Note that 
%\[d_y =\int_{\R}u e^{uy}\mu(du)  \]
%and 
%\[d_{yy} = \int_{\R}u^2 e^{uy}\mu(du) \]
%Then one can write 
%\begin{align*}
%&\frac{d}{d y} \left(  \ln(\frac{d}{d y} \mu_{0,y} (du)) - \frac{d}{d y} \ln(\pi_y) \right) \\
%=&\frac{d}{d y} \left( \ln(e^uy \mu(du)) +\ln(ud-d_y) -\ln(d^2) -\ln(d\int_{\theta_0}^{\infty}ue^{uy}\mu(du)-  \int_{\theta_0}^{\infty}e^{uy}\mu(du) d_y+\ln(d^2))\right)\\
%=&\frac{d}{d y} \left( \ln(e^uy ) +\ln(ud-d_y)  -\ln(d\int_{\theta_0}^{\infty}ue^{uy}\mu(du)-  \int_{\theta_0}^{\infty}e^{uy}\mu(du) d_y)\right)\\
%=& \frac{d_yu-d_{yy}}{du-d_y}+u- \frac{d\int^{\infty}_{\theta_0} u^2e^{uy}\mu(du)-d_{yy}\int^{\infty}_{\theta_0}e^{uy}\mu(du)}{d\int^{\infty}_{\theta_0}ue^{uy}\mu(du)-d_y\int^{\infty}_{\theta_0}e^{uy}\mu(du)}. 
% \end{align*}
straightforward calculations show that 
\begin{align*}
\frac{\partial^2}{\partial \pi^2} \left(\frac{e^{uy(0,\pi)}}{\int_{S}e^{uy(0,\pi)} \mu(du)}\right) & = \frac{e^{uy(0,\pi)}F(u)}{(\pi')^3 g^3},
\end{align*}
where 
\begin{align*}
F(u) =& u^2\left( g\int_{S^+}ue^{uy}\mu(du)-g'\int_{S^+}e^{uy}\mu(du)\right)\\
&+u\left( g''\int_{S^+}e^{uy}\mu(du)-g\int_{S^+}u^2e^{uy}\mu(du)\right)  \\
&+g'\int_{S^+}u^2e^{uy}\mu(du)- g'' \int_{S^+}u e^{uy}\mu(du).
\end{align*}
Note that $F$ is a quadratic function in $u$, and that the coefficient of $u^2$ is positive since 
\[g\int_{S^+}ue^{uy}\mu(du)-g'\int_{S^+}e^{uy}\mu(du)=
g^2
\Cov_{0,\pi}(\Theta, \mathds{1}_{\{\Theta > \theta_0\}})>0.\]
Consequently, the set $\{F<0\}$ is a bounded interval (possibly empty). 
%\begin{align*}
%\frac{\partial^2}{\partial \pi^2} \left(\frac{e^{uy(0,\pi)}}{\int_{\R}e^{uy(0,\pi)} \mu(du)}\right) & =\frac{e^{uy}} {d \pi_y^2}\left( \frac{d_yu-d_{yy}}{du-d_y}+u- \frac{d\int^{\infty}_{\theta_0} u^2e^{uy}\mu(du)-d_{yy}\int^{\infty}_{\theta_0}e^{uy}\mu(du)}{d\int^{\infty}_{\theta_0}ue^{uy}\mu(du)-d_y\int^{\infty}_{\theta_0}e^{uy}\mu(du)}\right)\\ 
%&=  \frac{e^{uy}F(u)}{{\pi_y}^3 d^3},
%\end{align*}
%where 
%\begin{align*}
%F(u) =& u^2( d\int^{\infty}_{\theta_0}ue^{uy}\mu(du)-d_y\int^{\infty}_{\theta_0}e^{uy}\mu(du))+u( d_{yy}\int^{\infty}_{\theta_0}e^{uy}\mu(du)-d\int^{\infty}_{\theta_0}u^2e^{uy}\mu(du))  \\
%&+d_y\int^{\infty}_{\theta_0}u^2e^{uy}\mu(du)- d_{yy} \int^{\infty}_{\theta_0}u e^{uy}\mu(du) \\
%=& u^2 \Cov_{0,y}(\Theta, \mathds{1}_{\Theta > \theta_0})+u \Cov_{0,y}(\Theta^2, \mathds{1}_{\Theta > \theta_0})+C
%\end{align*}
Moreover, since 
\[\pi=\frac{\int_{S^+}e^{uy(0,\pi)}\mu(du)}{g(y(0,\pi))}=1-\frac{\int_{S^-}e^{uy(0,\pi)}\mu(du)}{g(y(0,\pi))},\]
we have 
\begin{equation}
\label{lin}
\int_{S^+} \frac{\partial^2}{\partial \pi^2} \left(\frac{e^{uy(0,\pi)}}{\int_{\R}e^{uy(0,\pi)} \mu(du)}\right)\mu(du)=\int_{S^-}\frac{\partial^2}{\partial \pi^2} \left(\frac{e^{uy(0,\pi)}}{\int_{\R}e^{uy(0,\pi)} \mu(du)}\right)\mu(du)   =0.
\end{equation}
Therefore we must have
\[F(\theta_0)<0,\]
so the interval $\{F<0\}\not=\emptyset$. Denote the end-points of this interval by $u_0$ and $u_1$, respectively, so that $\{F<0\}=(u_0,u_1)$, with $\theta_0\in(u_0,u_1)$.
Then, using \eqref{decr} we find that
\begin{eqnarray*}
I_1&=&\int_{S\cap(-\infty,u_1)} \frac{e^{uy(0,\pi)}F(u)}{(\pi'(y(0,\pi)))^3 g^3(y(0,\pi))}
\E \left[H_1\left(\left.\frac{\alpha(X_1,\pi)}{\beta(X_1,\pi)}\right)\right\vert\Theta=u\right]     \mu(du)\\
&& + 
\int_{S\cap [u_1,\infty)} \frac{e^{uy(0,\pi)}F(u)}{(\pi'(y(0,\pi)))^3 g^3(y(0,\pi))}
\E \left[H_1\left(\left.\frac{\alpha(X_1,\pi)}{\beta(X_1,\pi)}\right)\right\vert\Theta=u\right]     \mu(du)\\
&\leq &
\E \left[H_1\left(\left.\frac{\alpha(X_1,\pi)}{\beta(X_1,\pi)}\right)\right\vert\Theta=u_1\right]    \int_{(\theta_0,u_1)} \frac{e^{uy(0,\pi)}F(u)}{(\pi'(y(0,\pi)))^3 g^3(y(0,\pi))}
 \mu(du) \\
&&+
\E \left[H_1\left(\left.\frac{\alpha(X_1,\pi)}{\beta(X_1,\pi)}\right)\right\vert\Theta=u_1\right]    \int_{[u_1,\infty)} \frac{e^{uy(0,\pi)}F(u)}{(\pi'(y(0,\pi)))^3 g^3(y(0,\pi))}
 \mu(du)\\
&=&0,
\end{eqnarray*}
where we used \eqref{lin} in the last equality.

Similarly, $\E \left[H_2\left(\frac{\alpha(X_1,\pi)}{\beta(X_1,\pi)}\right)\vert\Theta=u\right] $ increases in $u$, so
\begin{eqnarray*}
I_2&=&\int_{S^+\cap (-\infty,u_0]} \frac{e^{uy(0,\pi)}F(u)}{(\pi'(y(0,\pi)))^3 g^3(y(0,\pi))}
\E \left[H_2\left(\left.\frac{\alpha(X_1,\pi)}{\beta(X_1,\pi)}\right)\right\vert\Theta=u\right]     \mu(du)\\
&& + 
\int_{S^-\cap (u_0,\infty)} \frac{e^{uy(0,\pi)}F(u)}{(\pi'(y(0,\pi)))^3 g^3(y(0,\pi))}
\E \left[H_2\left(\left.\frac{\alpha(X_1,\pi)}{\beta(X_1,\pi)}\right)\right\vert\Theta=u\right]     \mu(du)\\
&\leq &
\E \left[H_2\left(\left.\frac{\alpha(X_1,\pi)}{\beta(X_1,\pi)}\right)\right\vert\Theta=u_0\right]   
\int_{S^-} \frac{e^{uy(0,\pi)}F(u)}{(\pi'(y(0,\pi)))^3 g^3(y(0,\pi))}
 \mu(du) \\
&=&0.
\end{eqnarray*}
Thus $\pi\mapsto\E_{0,\pi}[f(\Pi_1)]$ is concave.
\end{proof}

\begin{theorem}
\label{thm1}
The function $\pi \mapsto V(n, \pi)$ is concave for each fixed $n\geq 0$.
\end{theorem}

\begin{proof}
Define the cost function $V^N(n,\pi)$ as in \eqref{V}, but with the infimum being taken over stopping times $\tau\leq N-n$
($V^N$ is then the value function in a problem with a finite horizon). 
By an iterated use of Lemma~\ref{iteration} and Lemma~\ref{concav} and the fact that the minimum of two concave functions is concave, $\pi\mapsto V^N(n,\pi)$ is concave. Moreover, it is straightforward to check that $V^N(n,\pi)\to V(n,\pi)$ as $N\to\infty$,
and since the pointwise limit of concave functions is concave, the result follows.
\end{proof}

So far we have been working under the assumption that $\pi\in(0,1)$. One can further extend the value function $V$ to the boundary points $\pi \in \{0,1\}$ by setting $V(n,0) = V(n,1)=0$ for all $n$. In this way, $V$ is defined for every $\pi \in [0,1]$ and the concavity is preserved. 

In accordance with standard stopping theory, we introduce the continuation region $\C$ by
\[\C:= \{(n,\pi) \in \mathbb{N}_0\times [0,1]: V(n,\pi) < \pi \land (1-\pi)\},\] 
and the stopping region $\D$ by
\[\D:= \{(n,\pi) \in \mathbb{N}_0\times [0,1]: V(n,\pi) = \pi \land (1-\pi)\}.\] 
The stopping time 
\[\tau^*:= \inf \{k \geq 0: (n+k, \Pi_{n+k}) \in \D\}\]
is an optimal strategy for our testing problem. 

The concavity of the value function has important implications for the structure of the continuation region.

\begin{corollary}
There exist functions $b_1: \mathbb{N}_0 \to [0,\frac{1}{2}]$ and $b_2: \mathbb{N}_0 \to [\frac{1}{2},1]$ such that 
\[\C=\{(n,\pi)\in \mathbb N_0\times[0,1]: b_1(n) < \pi < b_2(n)\}.\]
\end{corollary}

\begin{proof}
Since $V(n,0)=V(n,1)=0$, we have $\{(n,0)\}\cup\{(n,1)\}\subseteq\D$. The result then follows from concavity of $\pi\mapsto V(n,\pi)$ and the piecewise linearity of $\pi\mapsto \pi\wedge(1-\pi)$.
\end{proof}

\begin{remark}
In view of the bijection in Lemma~\ref{lem}, the fact that time sections of the continuation region are intervals in the $(n,\pi)$-coordinates implies that also time sections of the continuation region expressed in $(n,y)$-coordinates are intervals.
This is a well-known result, see \cite{MS} (under somewhat different assumptions). 
%Barnett discusses the problem of testing the unknown parameter with any prior distribution under Bernoulli sampling in his thesis \cite{BNB}, and showed that for a fixed time $n$ the continuation region, the acceptance and rejection regions are all intervals. In \cite{C}, Cabilio gave an explicit solution of the boundary under a beta prior, and they are monotone in time. The proofs are based on the dynamical programming principle and Lipschitz property of the Bayes loss.
%
%In our approach, the existence of the stopping boundaries is in the $\pi-$axis. This is equivalent to that in the $y-$axis as in the literature due to Lemma \ref{lem}. Furthermore, it is a consequence of the Lipschitz property of the value function, which follows from the concavity.
\end{remark}

\section{Concentration of the posterior distribution}
\label{concentration}

Recall that the mass above $\theta_0$ of the posterior distribution remains constantly equal to $\pi$ along a $\pi$-level curve.
In this section we show that the posterior distribution becomes more concentrated around $\theta_0$ along a level curve.
This result, however natural it appears, seems to be new in the literature; for related results showing that the 
conditional variance of the mean-square estimate is a supermartingale, see \cite{EKV}.

\begin{theorem}
\label{shrinking}
If $a<\theta_0<b$, then
\[n\mapsto\P_{n,\pi} (\Theta \leq a) \quad \&\quad n\mapsto \P_{n,\pi} (\Theta> b)\]
are decreasing.
\end{theorem}

\begin{proof}
For the first claim, it suffices to show that 
\begin{equation}\label{ineq1}
\P_{0,\pi}(\Theta\leq a)\geq \P_{1,\pi}(\Theta\leq a)
\end{equation}
where $a<\theta_0$.
Moreover, without loss of generality, we may assume that $y(0,\pi)=0$ so that $\mu_{0,y(0,\pi)}=\mu$.
Let 
\[f(u):=e^{uy(1,\pi)-B(u)},\]
and let $S^a:= S\cap (-\infty,a]$. Note that
\[\P_{0,\pi} \left(\Theta\leq a\right) = \int_{S^a}\mu(du)\]
and
\[\P_{1,\pi} \left(\Theta \leq a\right) =\frac{\int_{S^a}f(u)\mu(du)}{\int_{S}f(u)\mu(du)}.\]
%Since 
%\[\frac{\int_{-\infty}^{\theta_0}f(u)\mu(du)}{\int_{-\infty}^{\infty}f(u)\mu(du)}=\pi=\int_{-\infty}^{\theta_0}\mu(du) ,\]
%for the inequality \eqref{ineq1} to hold, it suffices to show that
%\[\frac{\int_{-\infty}^a f(u) \mu(du)}{\int_a^{\theta_0}f(u) \mu(du)} \leq\frac{\int_{-\infty}^a\mu(du)}{\int_a^{\theta_0} \mu(du)}.\]
Also note that 
\[\frac{\partial f(u)}{\partial u} = f(u)(y(1,\pi)-B'(u)).\]
Therefore, since $B$ is convex, we have that $f$ changes its monotonicity (from increasing to decreasing) at most once.
Now we consider two separate cases: 
\begin{itemize} 
\item[(i)]
$f(a)\leq f(\theta)$
\end{itemize}
and
\begin{itemize}
\item[(ii)]
$f(a)>f(\theta)$.
\end{itemize}

If (i) holds, then $(f(u)-f(a))(u-a)\geq 0$ for $u\leq \theta_0$ (since $f$ changes its monotonicity at most once).
Consequently, if $\mu(S^-\setminus S^a)\not= 0$, then
\begin{equation}
\label{step1}
\frac{\int_{S^a} f(u) \mu(du)}{\int_{S^- \setminus S^a}f(u) \mu(du)}\leq \frac{f(a)\int_{S^a} \mu(du)}{f(a)\int_{S^- \setminus S^a}\mu(du)}=\frac{\int_{S^a}\mu(du)}{\int_{S^- \setminus S^a}\mu(du)}.
\end{equation}
(if $\mu(S^-\setminus S^a)= 0$, then \eqref{ineq1} holds trivially with equality).
Since 
\[\frac{\int_{S^-}f(u)\mu(du)}{\int_{S}f(u)\mu(du)}=1-\pi=\int_{S^-}\mu(du) ,\]
\eqref{step1} implies that
\begin{eqnarray*}
\P_{1,\pi} \left(\Theta \leq a\right) &=& \frac{\int_{S^a}f(u)\mu(du)}{\int_{S}f(u)\mu(du)}
=\frac{\int_{S^a}f(u)\mu(du)\int_{S^a}\mu(du)}{\int_{S^-}f(u)\mu(du)}\\
&=& \frac{ \frac{\int_{S^a}f(u)\mu(du)}{\int_{S^- \setminus S^a}f(u)\mu(du)} \int_{S^-}\mu(du)}
{1+\frac{\int_{S^a}f(u)\mu(du)}{\int_{S^- \setminus S^a}f(u)\mu(du)}}
\leq \int_{S^a}\mu(du) = \P_{0,\pi}(\Theta\leq a), 
\end{eqnarray*}
so \eqref{ineq1} holds.  

On the other hand, if (ii) holds, then the fact that $f$ changes its monotonicity at most once gives that 
$(f(u)-f(\theta_0))(u-\theta_0)\leq 0$ for all $u\geq a$. Consequently, 
\begin{equation}
\label{step2}
\frac{\int_{S^- \setminus S^a} f(u) \mu(du)}{\int_{S^+}f(u) \mu(du)}\geq \frac{f(\theta_0)\int_{S^- \setminus S^a}\mu(du)}{f(\theta_0)\int_{S^+}\mu(du)}=\frac{\int_{S^- \setminus S^a}\mu(du)}{\int_{S^+}\mu(du)}.
\end{equation}
Since
\[\frac{\int_{S^+}f(u)\mu(du)}{\int_S f(u)\mu(du)}=\pi=\int_{S^+}\mu(du),\]
the inequality \eqref{step2} yields
\begin{eqnarray*}
\P_{1,\pi} \left(\Theta > a\right) &=& \frac{\int_{S \setminus S^a}f(u)\mu(du)}{\int_{S}f(u)\mu(du)}
=\frac{\int_{S \setminus S^a}f(u)\mu(du)\int_{S^+}\mu(du)}{\int_{S^+}f(u)\mu(du)}\\
&=& \left(1+ \frac{\int_{S^- \setminus S^a}f(u)\mu(du)}{\int_{S^+}f(u)\mu(du)} \right)\int_{S^+}\mu(du)
\geq \int_{S \setminus S^a}\mu(du) = \P_{0,\pi}(\Theta> a), 
\end{eqnarray*}
from which \eqref{ineq1} follows.
Finally, the second inequality (for $b>\theta_0$) follows by a similar argument.
\end{proof}

As a consequence, we can show that the level curves are spreading out along the time axis.

\begin{corollary}
\label{lcspread}
Let $0<\pi_1<\pi_2<1$. Then $n\mapsto y(n, \pi_2)-y(n, \pi_1)$ is non-decreasing.
\end{corollary}

\begin{proof}
Recall from \eqref{cov} that 
\begin{eqnarray*}
\frac{\partial q}{\partial y}(n,y(n,\pi)) &=& \E_{n,\pi}[\Theta\mathds{1}_{\{\Theta> \theta_0\}} ]-\P_{n,\pi}(\Theta> \theta_0)\E_{n,\pi}[\Theta]\\
&=& (1-\pi) \E_{n,\pi}[\Theta\mathds{1}_{\{\Theta> \theta_0\}} ]-\pi \E_{n,\pi}[\Theta\mathds{1}_{\{\Theta\leq \theta_0\}} ]
\end{eqnarray*}
By Theorem~\ref{shrinking}, this covariance (with respect to the posterior distrbution)
is non-increasing in $n$. By this, the level curves are spreading out.
\end{proof}

\section{Conditions for monotonicity in time}
\label{mono1}

In this section we investigate whether $n\mapsto V(n,\pi)$ is non-decreasing. If this monotonicity holds, then the stopping boundaries
$b_1$ and $b_2$ will be non-decreasing and non-increasing, respectively.
To prove the monotonicity of $V$, we will use the following assumption.

\begin{assumption}\label{stochdom}
For any $\pi\in(0,1)$ and $n\geq m\geq 0$, the random variable $\Pi_{m+1}\vert\{\Pi_m=\pi\}$ dominates $\Pi_{n+1}\vert\{\Pi_n=\pi\}$ in convex order.
\end{assumption}

\begin{theorem}\label{mon}
Assume that Assumption~\ref{stochdom} holds. Then $V(n,\pi)$ is non-decreasing in $n$, and the boundaries $b_1$ and $b_2$ are thus non-decreasing and non-increasing, respectively.
\end{theorem}

\begin{proof}
For any concave function $f$, we have $\E_{m,\pi}[f(\Pi_{m+1})] \leq \E_{n,\pi}[f(\Pi_{n+1})]$. It thus follows from Lemma~\ref{iteration} that $V(n,\pi)$ is non-decreasing in $n$; the monotonicity of the boundaries $b_1$ and $b_2$ is a direct consequence.
\end{proof}

Since $\Pi_{m+1}\vert\{\Pi_m=\pi\}$ and $\Pi_{n+1}\vert\{\Pi_n=\pi\}$ have the same expected value $\pi$, a sufficient condition for stochastic domination in convex order is that there exists a point $\pi_0$ around which the distribution of $\Pi_{n+1}\vert\{\Pi_n=\pi\}$
is more concentrated compared to the distribution of $\Pi_{m+1}\vert\{\Pi_m=\pi\}$ in the sense that 
\[\P_{m,\pi}(\Pi_{m+1}\leq \alpha) \geq \P_{n,\pi}(\Pi_{n+1}\leq \alpha) \]
and 
\[\P_{m,\pi}(\Pi_{m+1}>\beta)\geq  \P_{n,\pi}(\Pi_{n+1}\beta) \]
for $\alpha< \pi_0<\beta$. 
Using Theorem~\ref{shrinking}, we now
state a sufficient condition under which the above concentration property holds. 

\begin{theorem}
\label{egmonotone}
Assume that the observations are continuously distributed with density of the form
\[h(x) p_u(x), \]
with $\nu(dx) = h(x) dx$ for some nonnegative continuous function $h$ such that $I:=\{h>0\}$ is an interval. 
Assume that either 
\begin{itemize}
\item[(i)]
$h(x) p_u(x)$ is increasing in $x$ on $I$, and $S^+$ is a singleton
\end{itemize}
or
\begin{itemize}
\item[(ii)]
$h(x) p_u(x)$ is decreasing in $x$ on $I$, and $S^-$ is a singleton
\end{itemize}
holds.
Then Assumption~\ref{stochdom} holds, so $n\mapsto V(n,\pi)$ is non-decreasing.
\end{theorem}

\begin{proof}
We will verify Assumption~\ref{stochdom} for $m=0$ and $\pi=\mu(S\cap[0,\infty))$; the general case follows by translation.
Thus we consider two distributions $\mu$ and $\mu':=\mu_{n,y(n,\pi)}$.
Let $F(a)=\P_{0,\pi}(\Pi_1\leq a)$ and $G(a)=\P_{n,\pi}(\Pi_{n+1}\leq a)$.
Since $F$ and $G$ are continuous distribution functions with the same expected value $\pi$, there exists
$\pi_0\in(0,1)$ with $F(\pi_0)=G(\pi_0)$. We claim that 
\begin{equation}
\label{claim}
F'(\pi_0)< G'(\pi_0)
\end{equation}
at such a point (unless $\mu$ is a two-point distribution), which implies that there is a single intersection point $\pi_0$ and that $G$ is more concentrated about 
$\pi_0$ in comparison with $F$.

To prove the claim, let $\pi_0$ be such that $F(\pi_0)=G(\pi_0)$, and denote by $x_1$ and $x_{n+1}$ the unique values such that
$q(1,x_1)=\pi_0=q(n+1,y(n,\pi)+x_{n+1})$, so that $(1,x_1)$ and $(n+1,y(n,\pi)+x_{n+1})$ are both on the $\pi_0$-level curve.
We then have
\[F(\pi_0) = \int^{x_1}_{-\infty} \int_{S}h(x) p_u(x) \mu(du) dx\]
and
\[G(\pi_0) = \int^{x_{n+1}}_{-\infty} \int_{S}h(x) p_u(x) \mu'(du) dx,\]
so
\[F'(\pi_0)=\frac{\int_{S}h(x_1) p_u(x_1) \mu(du)}{\frac{\partial q}{\partial y}(1,x_1)}\]
and 
\[G'(\pi_0)=\frac{\int_{S}h(x_{n+1}) p_u(x_{n+1}) \mu'(du)}{\frac{\partial q}{\partial y}(n+1,y(n,\pi)+x_{n+1})}.\]
Note that
\[\frac{\partial q}{\partial y}(1,x_1)=Cov_{1,\pi_0}(\Theta,1_{\{\Theta\leq \theta_0\}})
\geq Cov_{n+1,\pi_0}(\Theta,1_{\{\Theta\leq \theta_0\}})=\frac{\partial q}{\partial y}(n+1,y(n,\pi)+x_{n+1})\]
since the covariance decreases along the $\pi_0$-level curve. Furthermore, the covariance is strictly decreasing along the level curve
(unless $\mu$ has support on only two points). 
Therefore, it suffices to show that $\int_{S}h(x_1) p_u(x_1)\mu(du)\leq \int_{S}h(x_{n+1}) p_u(x_{n+1})\mu'(du)$. 

To do that, 
assume that (i) holds so that $\mbox{supp } \mu\cap S^+= \{\theta_1\}$ for some $\theta_1>\theta_0$. 
Then $\mu$ and both $\mu'$ are identical on $S^+$, so it follows from Theorem~\ref{shrinking} that $\mu'$ stochastically dominates $\mu$. Consequently, 
\[\int^{x}_{-\infty} \int_{S}h(y) p_u(y) \mu(du) dy\geq \int^{x}_{-\infty} \int_{S}h(y) p_u(y) \mu'(du) dy\]
for any $x$, so $F(\pi_0)=G(\pi_0)$ implies $x_1\leq x_{n+1}$.
Moreover, the relation $p(1,x_1)=p(n+1,y(n,\pi)+x_{n+1})$ implies that
\begin{equation}
\label{rel}
\frac{\int_{S^-} h(x_1) p_u(x_1)\mu(du)}{ h(x_1)p_{\theta_1}(x_{1})}=\frac{\int_{S^-} h(x_{n+1}) p_u(x_{n+1})\mu'(du)}{h(x_{n+1})p_{\theta_1}(x_{n+1})}.\end{equation}
Since the density function is increasing in $x$, we have $h(x_1)p_{\theta_1}(x_1)\leq h(x_{n+1})p_{\theta_1}(x_{n+1})$.
Equation \eqref{rel} thus implies
\begin{equation}
\label{ineq}
\int_{S}h(x_1) p_u(x_1)\mu(du)\leq \int_{S}h(x_{n+1}) p_u(x_{n+1})\mu'(du).
\end{equation}
If instead (ii) holds, then $x_1\geq x_{n+1}$, and \eqref{ineq} is derived in a similar way.

It follows that $F$ and $G$ have a unique intersection point $\pi_0$ (unless the support of $\mu$ has only two points, in which $F\equiv G$), and that $G$ is more concentrated about $\pi_0$ compared to $F$.
Consequently, the random variable $\Pi_1\vert\{\Pi_0=\pi\}$ dominates $\Pi_{n+1}\vert\{\Pi_n=\pi\}$ in convex order.
\end{proof}

\begin{example}{\bf (Exponential observations.)}
Assume that $\{X_k, k\geq 1\}$ are exponentially distributed with unknown intensity $\Theta$
and independent (conditional on $\Theta$) so that
\[\P(X_1\leq x\vert \Theta=u)=1-e^{-ux}.\]
This is not on the exponential form \eqref{expfam}, but it is straightforward to check that if one instead 
considers $X_k':=-X_k$, then $X_k'$ has density 
\[h(x) p_u(x) =\left\{\begin{array}{ll}
 \exp\{u x+\log u\} & x<0\\
0  & x\geq 0\end{array}\right.\]
with respect to Lebesgue measure.
For $u>0$, this density is increasing in $x$ on $I=(-\infty,0)$. Consequently, if $S^+$ is a singleton, then 
(i) in Theorem~\ref{egmonotone} gives that $V(n,\pi)$ is increasing in $n$, which leads to the monotonicity of the stopping boundaries.
\end{example}

\begin{example}{\bf (Gaussian observations with unknown variance.)}
If $\{X_k, k\geq 1\}$ are normally distributed with mean 0 and unknown standard deviation $\Theta$
and independent (conditional on $\Theta$), then the random variables
$X'_k=-\frac{1}{2} (X_k)^2$ are on the exponential form \eqref{expfam} with respect to the unknown variable $\Theta':=\Theta^{-2}$, with density 
\[h(x) p_{u}(x)=\left\{\begin{array}{ll}
\frac{2}{\sqrt{-\pi x}}\exp\left\{ux +\frac{1}{2}\log u \right\}  & x<0\\
0 & x\geq 0\end{array}\right.\]
with respect to Lebesgue measure.
Note that this density is increasing in $x$. 
Also note that $H_0$ holds precisely when $\Theta'\geq \theta_0^{-2}$.
Therefore, the value function is decreasing in time for any prior distribution $\mu$ such that $S^-$ is a singleton.
\end{example}

\begin{remark}
Theorem \ref{egmonotone} considers random variables that are continuously distributed on an interval. However, 
a closer inspection of the proof reveals that 
one can relax this assumption and instead assume that $H:=\supp(h)$ is the union of disjoint intervals.
Moreover, when these intervals become small, using approximation arguments one would expect a similar result for discrete distributions; we leave the details, as well as the precise formulation, of such a result.
\end{remark}

\section{Further discussion on time monotonicity}
\label{mono2}

While the conditions in Theorem~\ref{egmonotone} may appear somewhat restrictive, we have failed to remove the conditions. 
On the other hand, we have also been unable to produce examples within the exponential family for which the asserted time monotonicity fail. In this final section, we show that time monotonicity holds for arbitrary prior distributions $\mu$ in a few particular examples, see Sections 6.1-3 below. Based on these findings, we formulate the following conjecture.

\begin{conjecture}
The function $V(n,\pi)$ in \eqref{V} is non-decreasing in $n$ for any prior distribution $\mu$ and any $\nu$.
\end{conjecture}

\subsection{Gaussian observations with unknown mean}
Assume that the Gaussian sequence $\{X_k,k\geq 1\}$ has unknown mean $\Theta$ and known standard deviation $1$ so that the density (conditional on $\Theta=u$) is
\[\frac{1}{\sqrt{2\pi}}e^{-\frac{x^2}{2}}\exp\{ux-B(u)\},\]
where $B(u) =u^2/2$. In this case, the discrete time $\Pi$-process can be embedded in the corresponding continuous time process, as studied in \cite{EV}. Moreover, the discrete time problem then corresponds to the continuous time problem in \cite{EV} but with the restriction that stopping is only allowed at integer times. For such a problem, the techniques used in \cite{EV} (in particular, preservation of concavity for
martingale diffusions coupled with time-decay of the diffusion coefficient of $\Pi$) show that $V(n,\pi)$ is non-decreasing in $n$.

\subsection{Bernoulli observations}
\label{bern}
Consider a sequence $\{X_k,k\geq 1\}$ which is Bernoulli distributed with parameter $\Theta$ so that 
\[\P(X_k=1\vert \Theta=u)=1-\P(X_k=0\vert \Theta=u)=u.\]
(This is on the exponential form if one instead uses $\Theta':=\log\frac{\Theta}{1-\Theta}$ as the unknown parameter, because then
\[\P(X_k=x\vert \Theta'=u)= \P(X_k=x\vert \Theta=\frac{e^u}{1+e^u})=e^{ux-\log(1+e^u)}\]
for $x\in\{0,1\}$.) Also assume that the distribution of $\Theta$ is $\mu$, which is a given measure on $[0,1]$.

Note that since observations are binary, $\Pi_{n+1}$ can only take two different values if started at a given point $(n,\pi)$.
For $m\leq n$ we have 
\begin{align*}
\P_{m,\pi}(\Theta > \theta_0| X_{m+1} = 1) & = \frac{\int_{S^+} u \mu_{m,y(m,\pi)}(du) }{\int_S u \mu_{m,y(m,\pi)}(du)} 
=  \frac{1 }{1+\frac{\E_{m,\pi}[\theta \mathds{1}_{\{\theta\leq \theta_0\}}]}{\E_{m,\pi}[\theta \mathds{1}_{\{\theta> \theta_0\}}]}}\\
& \geq  \frac{1 }{1+\frac{\E_{n,\pi}[\theta\mathds{1}_{\{\theta\leq \theta_0\}}]}{\E_{n,\pi}[\theta \mathds{1}_{\{\theta >\theta_0\}}]}}
=\P_{n,\pi}(\Theta> \theta_0| X_{n+1} = 1),
\end{align*}
where the inequaility is a consequence of Theorem~\ref{shrinking}.
Similarly, 
\begin{align*}
\P_{m,\pi}(\Theta > \theta_0| X_{m+1} = 0) &= \frac{1 }{1+\frac{\E_{m,\pi}[(1-\theta) \mathds{1}_{\{\theta\leq \theta_0\}}]}{\E_{m,\pi}[(1-\theta) \mathds{1}_{\{\theta>\theta_0\}}]}} \\
&\leq \frac{1 }{1+\frac{\E_{n,\pi}[(1-\theta) \mathds{1}_{\{\theta\leq \theta_0\}}]}{\E_{n,\pi}[(1-\theta) \mathds{1}_{\{\theta>\theta_0\}}]}}=
\P_n(\Theta > \theta_0| X_{n+1} =0) .
\end{align*}
Since two two-point distributions with the same mean and with mass on $\{a,b\}$ and $\{a',b'\}$, where $a\leq a'<b'\leq b$, are ordered in convex order, it follows that the distribution of $\Pi_{m+1}$ under $\P_{m,\pi}$ dominates 
the distribution of $\Pi_{n+1}$ under $\P_{n,\pi}$ in convex order. Thus Assumption~\ref{stochdom} is satisfied, so 
time monotonicity for arbitrary priors $\mu$ holds by Theorem~\ref{mon}.

\subsection{Binomial observations}
Consider a general prior distribution $\mu$ on $[0,1]$ for $\Theta$, and observations $\{X_k,k\geq 1\}$ that are 
$Bin(N, \Theta)$. As in Section~\ref{bern}, this is on the exponential form if one uses $\Theta' = \log\frac{\Theta}{1-\Theta}$
as the unknown parameter.

Now consider a sequential testing problem for Bernoulli observations with the same unknown parameter $\Theta$, but
where the cost $c$ is only imposed on the $(Nk+1)$-th observation, for all $k\in \mathbb N_0$.
Denoting the value function of that sequential problem by $V^{Ber}(n,\pi)$, arguments similar to those in Section~\ref{bern}
imply that the function $k\mapsto V^{Ber}(kN,\pi)$ is non-decreasing. 
%Define
%\[\Pi_k^Y := \P(\Theta > \theta_0| \mathcal{F}_k^Y).\]
%Denote the posterior probability process for a sequence of Bernoulli observations with the same unknown parameter $\Theta$ as $\Pi^X$, and observe that for all $k\in \mathbb{Z}_+$ and all $y\in \mathbb{Z}$,
%\[\Pi_{Nk}^X\vert_{\sum_{i=1}^{Nk} X_i=y}=\Pi_{k}^Y\vert_{\sum_{i=1}^{k} Y_i=y}.\]
%That is to say, with the same sufficient statistic $y$, the posterior probability process with the Bernoulli observations is the posterior probability process with the binomial observations scaled by factor $N$ on the time axis. As a consequence, we can consider the testing problem where we have Bernoulli observations and the discrete cost $c$ is only imposed on the $(Nk+1)$-th observation, for all $k\in \mathbb{Z}_+$. T
However, the value function $V(n,\pi)$ for Binomial observations clearly coincide with $V^{Ber}(nN,\pi)$, so 
time monotonicity holds also for $V(n,\pi)$.

\end{document}